\documentclass[12pt]{article}

\title{Asymptotic distribution of odd balanced unimodal sequences with rank congruent to $a$ modulo $c$}

\author{Taylor Garnowski}

\usepackage{amssymb,amsmath,amsthm, amsfonts}
\usepackage{comment}
\usepackage{mathrsfs}
\usepackage{float}
\usepackage[dvipsnames]{xcolor}
\usepackage{amsmath}
\usepackage{graphicx}
\usepackage[ backend = biber, style = numeric, maxnames= 50]{biblatex}
\renewbibmacro{in:}{}
\usepackage{mathtools}

\usepackage{authblk}
\usepackage[all,knot]{xy}
\usepackage{tabularx}
\usepackage{setspace}
\usepackage{todonotes}
\usepackage{tikz}

\newtheorem{theorem}{Theorem}

\newtheorem{lemma}[theorem]{Lemma}

\newtheorem{corollary}[theorem]{Corollary}
\newtheorem{remark}[theorem]{Remark}

\newtheorem{proposition}[theorem]{Proposition}

\theoremstyle{definition}

\begin{document}
\maketitle
\begin{abstract}
    We compute an asymptotic estimate for odd balanced unimodal sequences for ranks congruent to $a \pmod{c}$ for $c>1$ odd. We find the interesting result that the odd balanced unimodal sequences are asymptotically related to the overpartition function. This is in contrast to strongly unimodal sequences which, are asymptotically related to the partition function. Our proofs of the main theorems rely on the representation of the generating function in question as a mixed mock modular form. 
\end{abstract}
\section{Introduction and statement of results}
\hspace{5mm} A unimodal sequence of size $n$ is a sequence of positive integers, $\{a_j\}^t_{j=1}$, whose sum is $n$ and whose parts satisfy
\begin{align}\label{unimodalcondition}
    a_1\leq a_2\leq...\leq a_c\geq a_{c+1}\geq a_{c+2}\geq...\geq a_t.
\end{align}
The part $a_c$ (which need not be unique) is referred to as a \textit{peak}. In particular, unimodal sequences are partitions of $n$, and if we replace $\leq$ and $\geq$ with strict inequalities in Eq. \eqref{unimodalcondition}, we arrive at the \textit{strongly} unimodal sequences, which then have a unique peak $a_c$. In both the general and strict case, we can define a rank statistic, usually denoted by $m$, which is defined as
\begin{align*}
    m:= \textnormal{number of parts after the peak}-\textnormal{number of parts before the peak}. 
\end{align*}
Such sequences are denoted $u(m,n)$. It is clear that $m$ need not be positive. While interesting as a combinatorial object, the real story of strict unimodal sequences is hidden in the generating function, given by \cite{Bringmannbook},
\begin{align*}
    U(w;q):= \sum_{n\geq 1}(-wq,-w^{-1}q;q)_{n-1}q^n,
\end{align*}
where $w:= e^{2\pi i z}$ for $z\in \mathbb{C}$ and $q:= e^{2\pi i \tau}$ with $\tau \in \mathbb{H}$, and  $(A,B;q)_n := (A;q)_n(B;q)_n$, where $(A;q)_n$ is the usual $q$-Pochammer symbol with $n\in  \mathbb{N}\cup\{\infty\}$. Many authors have studied this generating function  for its analytic properties  \cite{unimodalwright2016,overlinedunimodal,unimodalprobability,unimodalmock2012}. One of the profound features of the generating function $U(w;q)$ is its relation to analytic functions on the real line. The authors of \cite{unimodalmock2012} showed that scaled versions of the function $U(1;q)$ can be defined as functions on $\mathbb{Q}$, and these functions are \textit{quantum modular forms} in the sense of Zagier \cite{Quatummodular}. The quantum modular forms are those functions whose obstruction to having a modular type transformation on $\mathbb{Q}$ for some subgroup of $\textnormal{SL}_2(\mathbb{Z})$ can be characterized by a function that has an analytic extension to an open subset of the real line. More interesting for this work is the relation between the $u(m,n)$ and the classical partition function $p(n)$. It is well-known since the work of Hardy and Ramanujan \cite{hardyramanujan} that for large $n$, $p(n)$ satisfies
\begin{align*}
    p(n)\sim \frac{1}{4\sqrt{3} n}e^{\pi\sqrt{\frac{2n}{3}}}.
\end{align*}
Rhoades proved in 2014 an asymptotic formula for the number of strongly unimodal sequences $u(n)$ \cite{rhoadesunimodal}, and the authors of \cite{peakposition} extended this result for fixed rank $m$. The main result in \cite{peakposition} reads
\begin{align}\label{u(n)top(n)}
    u(m,n) \sim \frac{p(n)}{4}.
\end{align}
In light of the special analytic properties of $U(w;q)$ , Kim, Lim, and Lovejoy  defined \textit{odd-balanced} unimodal sequences of rank $m$ by interpreting the following generating function,
\begin{align}\label{Vgenfunction}
    V(w;q):= \sum_{n\geq 0}\frac{(-wq,-qw^{-1};q)_nq^n}{(q;q^2)_{n+1}} = \sum_{\substack{n\geq 0\\m\in \mathbb{Z}}}v(m,n)w^mq^n.
\end{align}
The coefficients $v(m,n)$ count the number of  unimodal sequences of size $2n+2$ and rank $m$, where the peak is even, the odd parts can repeat, but must be identical on each side of the peak, and the rest of the members of the sequence satisfy strict inequalities. A few  examples of odd-balanced sequences of size $12$ are $(1,1,2,4,2,1,1), (1,3,4,3,1), (12)$, and $(1,8,2,1)$. The authors of \cite{kimlimlovejoy} showed that $V(1;q^{-1})$ exhibits quantum modular properties analogous to the generating function for strongly unimodal sequences.
\par 
Let $v(a,c;n)$ denote the same count as $v(m,n)$ but with the relaxed condition that the rank is congruent to $a\pmod{c}$. Let $\zeta^a_c:= e^{2\pi i \frac{a}{c}}$. Using orthogonality of roots of unity, we have that the generating function for the $v(a,c;n)$, $V(a,c;q)$, can be written formally as
\begin{align}\label{orthogonality}
    V(a,c;q) &:= \frac{V(1;q)}{c}+ \frac{1}{c}\sum^{c-1}_{j=1}\zeta^{-aj}_cV(\zeta^j_c;q).
\end{align}
\par It is not hard to see that the $v(a,c;n)$ are weakly increasing in $n$. To see this, notice that $v(a,c;n+1)$ counts sequences of size $2(n+1)+2 = 2n+4$. Therefore, we can take each of the sequences of size $2n+2$ and add a $1$ to each side of the peak without changing the rank. That is,
\begin{align}\label{monotone}v(a,c;n)\leq v(a,c;n+1).
\end{align}
Our main theorem involves the $v(a,c;n)$ and describes how they are distributed with respect to the overpartition function $\overline{p}(n)$. For an account of overpartitions and the combinatorics that accompany them, we refer the reader to \cite{overpartitions}. Due to the modularity of the generating function for the $\overline{p}(n)$, an analogous asymptotic result to the $p(n)$ follows from Hardy and Ramanujan \cite{hardyramanujan}:
\begin{align*}
    \overline{p}(n) \sim \frac{1}{8n}e^{\pi\sqrt{n}}.
\end{align*}
In an attempt to find an analogy to Eq. \eqref{u(n)top(n)} for $v(m,n)$, we find and prove the following for $v(a,c;n)$.

\begin{theorem}\label{maintheorem} Let $c>1$ be odd. Then as $n\to\infty$,
\begin{align*}
    v(a,c;n) \sim \frac{1}{16cn^{\frac{3}{4}}}e^{\pi\sqrt{n}}\sim \frac{n^{\frac{1}{4}}}{2c}\overline{p}(n).
\end{align*}
\end{theorem}
\begin{remark}\normalfont
There are two things worth noting here:
\begin{enumerate}
\setlength\itemsep{1mm}
\item The exclusion of even $c$ is related to the fact that the generating function, $V(w;q)$, is not a mixed mock modular form at $z = \frac{1}{2}$. Mixed mock modular forms are essentially sums of products of mock modular forms and weakly holomorphic modular forms. The representation as a mixed mock modular form is advantageous because mock modular forms have modular completions to modular objects, which then allows one to compute accurate expansions of mixed mock modular forms near the real line. We will explore this in more detail in the coming sections.
\item The result above can be interpreted as an equidistribution result with respect to the modulus $c$. That is,
\begin{align*}
    v(a,c;n) \sim \frac{v(n)}{c}.
\end{align*}
Results of this type were explored by Ciolan for the overpartition function \cite{alexoverpartitions} and by Males for the partition function \cite{joshpartitions}.
\end{enumerate}
\end{remark}
An immediate consequence of Theorem \ref{maintheorem} involves a log-concavity type bound.
\begin{corollary}\label{logconcave}
Let $c>1$ be odd. Then there exists an $N$, such that for all $n>N$,
\begin{align*}
  v(a,c,2n)\leq v(a,c;n-1)v(a,c;n+1) < \sqrt{n}\cdot \overline{p}(n-1)\overline{p}(n+1).
\end{align*}
\end{corollary}
\par 
The rest of this work is organized as follows. In Section \ref{basic objects}, we collect some facts and prove some lemmas related to modular forms and mixed mock modular forms. In Section \ref{zeq0sec}, we prove a growth result for the $v(n)$, which is the main term in our proof of Theorem \ref{maintheorem}. In Section \ref{zfixedrootsofunity}, we calculate the growth of the generating function $V(w;q)$ for $w$ a fixed root of unity. In Section \ref{proofsection}, we prove Theorem \ref{maintheorem} and Corollary \ref{logconcave}. We conclude this work in Section \ref{conclusionandfringecases} by summarizing our results, and offering some discussion on the cases of $c$ even.

\section{Basic objects and useful facts}\label{basic objects}
\hspace{5mm} The key step in our work is the fact that the generating function, $V(w;q)$, can be written as a sum of products of modular forms and mock modular forms. This type of object can be loosely referred to as a \textit{mixed mock modular form}. As Zwegers showed in his thesis \cite{Zwegers}, the building blocks of mock modular forms are the \textit{Appell sums}, and the building blocks of mixed mock modular forms are the \textit{higher level Appell sums} \cite{Zwegers2}. Recall that for $u, v \in \mathbb{C}$ the higher level Appell sum from \cite{Zwegers2} is defined by
\begin{align*}
    A_{\ell}(u,v,\tau) := e^{\pi i \ell u}\sum_{n\in \mathbb{Z}}\frac{(-1)^{\ell n}e^{2\pi i n v}q^{\frac{\ell\cdot n(n+1)}{2}}}{1-e^{2\pi i u}q^n}.
\end{align*}
Note that $A_1$ is the classical Appell sum, whose (mock) modular properties were studied at length in \cite{Zwegers}.  In particular,
\begin{align}\label{Appelltransnonnormal}
    -\frac{1}{\tau}e^{\frac{\pi i(u^2-2uv)}{\tau}}A_1\left(\frac{u}{\tau}, \frac{v}{\tau}; -\frac{1}{\tau}\right)+ A_1(u,v;\tau) = \frac{1}{2i}h(u-v;\tau)\vartheta(v;\tau),
\end{align}
where $h(z;\tau)$ is the Mordell integral defined by
\begin{align*}
    h(z;\tau):= \int^{\infty}_{-\infty}\frac{e^{\pi i\tau x^2-2\pi z x}}{\textnormal{cosh}(\pi x)}dx.
\end{align*}
The Mordell integral satisfies its own transformation law given by
\begin{align}
    h\left(\frac{z}{\tau}; -\frac{1}{\tau}\right) = \sqrt{-i\tau}e^{-\frac{\pi i z^2}{\tau}}h(z;\tau).
\end{align}
Knowing how the Mordell integral behaves for $\tau$ near the real line is essential for our study here. In light of this, we first prove a simple lemma related to the positivity of the Mordell integral for certain $z$ and $\tau$ of small modulus. 
\begin{lemma}\label{lowerboundmordell}
Let $-\frac{1}{2}<z<\frac{1}{2}$ be a real variable. Then as $\tau \to 0$ within $\mathbb{H}$, 
\begin{align*}
    0< h(z;\tau)\ll 1.
\end{align*}
We also have the specific value
\begin{align*}
    h(0;0)=1.
\end{align*}
\begin{proof}
By definition,
\begin{align*}
    \left|h(z;\tau)\right| \leq \int^{\infty}_{-\infty}\left|\frac{e^{\pi i\tau x^2-2\pi z x}}{\textnormal{cosh}(\pi x)}\right|dx\leq \int^{\infty}_{-\infty}\frac{e^{-2\pi z x}}{\textnormal{cosh}(\pi x)}dx,
\end{align*}
where the last function is integrable for $z$ in the specified range. Therefore, by dominated convergence
\begin{align*}
    \lim_{\tau\to 0}h(z;\tau) = h(z;0) = \int^{\infty}_{-\infty}\frac{e^{-2\pi z x}}{\textnormal{cosh}(\pi x)}dx.
\end{align*}
When $z = 0$, we have
\begin{align*}
    h(0;0) = \int^{\infty}_{-\infty}\frac{1}{\textnormal{cosh}(\pi x)}dx = 1,
\end{align*}
which completes the proof.\end{proof}
\end{lemma}
\par
For $z\in\mathbb{C}$, the Jacobi theta function (or $\vartheta$-function for short) is defined by \cite{Zwegers},
\begin{align*}
    \vartheta(z;\tau):= \sum_{n\in \frac{1}{2}+\mathbb{Z}}e^{\pi i n^2\tau+2\pi i n\left(z+\frac{1}{2}\right)}.
\end{align*}

We will use the following identities frequently in this work.
\begin{proposition}[See Ch. 1 of \cite{Zwegers}]\label{transformprop} The following formulae hold\textnormal{:}
\begin{enumerate}
    \item $\vartheta(z+1;\tau) = -\vartheta(z;\tau)$,
    \item $\vartheta(z+\tau;\tau) = -e^{-\pi i \tau-2\pi i z}\vartheta(z,\tau)$,
    \item $\vartheta(z;\tau+1) = e^{\frac{\pi i}{4}}\vartheta(z;\tau)$,
    \item\label{thetatransitem}
        $\vartheta\left(\frac{z}{\tau}; -\frac{1}{\tau}\right) = -i\sqrt{-i\tau}e^{\frac{\pi i z^2}{\tau}}\vartheta(z; \tau)$,
        and
    \item $\eta(\tau) = \frac{1}{\sqrt{-i\tau}}\eta\left(-\frac{1}{\tau}\right)$ and $\eta(\tau +1) = e^{\frac{\pi i}{12}}\eta(\tau)$, where $\eta$ is Dedekind's eta function,
    \begin{align*}
        \eta(\tau) := q^{\frac{1}{24}}(q)_{\infty}.
    \end{align*}
\end{enumerate}
\end{proposition}

Using the formulae in Proposition \ref{transformprop}, one can arrive at the following.

\begin{lemma}[See Lemma 3.8 of \cite{mebailey}]\label{thetatestimateprelim}
Let $\alpha \in [0,1)$ and $q_0 := e^{-\frac{2\pi i }{\tau}}$. Let $k > 1$ be a rational number. As $\tau \to 0$,
\begin{align}
     \vartheta(\alpha\tau; \tau) &=\label{fulllatticethetadecay} \frac{-2i\; \textnormal{sin}(\pi\alpha)q^{-\frac{\alpha^2}{2}}q^{\frac{1}{8}}_0}{\sqrt{-i\tau}}\left( 1 +  O(q_0) \right) ,\\
     \vartheta\left(\frac{1}{k} + \alpha\tau; \tau\right) &= \label{mixmibabi1moretime} -\frac{q^{-\frac{\alpha^2}{2}}e^{\pi i\alpha(1-\frac{2}{k})}}{\sqrt{-i\tau}}q^{\frac{1}{2k^2}-\frac{1}{2k}+\frac{1}{8}}_0\left(1 + O\left(q^{\frac{1}{k}}_0\right)  \right),\\
     \eta(\tau)&= \label{etagrowth}\frac{q^{\frac{1}{24}}_0}{\sqrt{-i\tau}}\left(1+O(q_0)\right).
\end{align}
\end{lemma}.

We conclude this section with a Tauberian theorem, which we state in a modern form with additional technical restraints shown in \cite{TauberianSchaffer}.
\begin{theorem}[See Theorem 1.1 of \cite{TauberianSchaffer} with $\alpha =0$]\label{Tauberian}
Let $c(n)$ denote the coefficients of a power series $C(q):= \sum^{\infty}_{n=0}c(n)q^n$ with radius of convergence equal to $1$. Define $\sigma:= x+iy \in \mathbb{C}$ with $x>0$. If the $c(n)$ are non-negative, are weakly increasing, and we have as $t\to 0^+$ that
\begin{align*}
    C\left(e^{-t}\right) \sim \lambda t^{\alpha}e^{\frac{A}{t}},
\end{align*}
and if for each $M>0$ such that $|y|\leq M|x|$,
\begin{align}\label{extraconditiontaub}
   C\left(e^{-\sigma}\right) \ll   |\sigma|^{\alpha}e^{\frac{A}{|\sigma|}}
\end{align}
with $A>0$ holds, then as $n\to \infty$
\begin{align*}
    c(n)\sim \frac{\lambda A^{\frac{\alpha}{2}+\frac{1}{4}}}{2\sqrt{\pi}n^{\frac{\alpha}{2}+\frac{3}{4}}}e^{2\sqrt{An}}.
\end{align*}
\end{theorem}

\begin{remark}
The authors of \cite{TauberianSchaffer} pointed out that the extra condition in Eq. \eqref{extraconditiontaub} is trivially satisfied if the function in question satisfies
\begin{align*}
     C(e^{-\sigma}) \sim \lambda \textnormal{Log}\left(\frac{1}{\sigma}\right)^{\alpha}\sigma^{\beta}e^{\frac{A}{\sigma}}
\end{align*}
in the specified angular region.
\end{remark}

\section{Analytic properties of the generating function $V(w;q)$}\label{analysis}
\hspace{5mm} As we mentioned in the previous section, our proof relies on the fact that $V(w;q)$ can be written as a mixed mock modular form for fixed $z\neq \frac{1}{2}$. We call a function $H:\mathbb{H} \to \mathbb{C}$ \textit{mixed mock modular of weight} $k$ for some subgroup $\Gamma \subset \textnormal{SL}_2(\mathbb{Z})$ if we can write it as as finite sum,
\begin{align}\label{mixedmockeq}
    H(\tau) = \sum^r_{j=1}F_j(\tau)M_j(\tau),
\end{align}
where $F_j$ and $M_j$ are modular (resp. mock) modular forms on $\Gamma$ of weight $f_j$ (resp. $m_j$) such that $f_j+m_j = k$. Kim, Lim, and Lovejoy in \cite{kimlimlovejoy} used the theory of indefinite theta series and a corresponding result of Mortenson and Osburn \cite{splittingmort} to write
\begin{align}\label{Vincomponents}
    \left(1+w^{-1}\right)qV(w;q) = -T_1(w;q)+T(w;q)-wT_2(w;q),
\end{align}
where
\begin{align*}
    T_1(w;q) &:= -iq^{\frac{1}{8}}w^{-\frac{1}{2}}\mu\left(z+\frac{1}{2},\frac{1}{2},\tau\right),\\
    T(w,q)&:= -q^{-\frac{1}{8}}w^{-\frac{1}{2}}\frac{\vartheta(\frac{1}{2}+z;\tau)}{\vartheta(\tau;2\tau)}\mu\left(2z+\frac{1}{2},\frac{1}{2};2\tau\right),\\
    T_2(w;q)&:= iq^{\frac{11}{8}}w^{-\frac{1}{2}}\frac{\vartheta(4\tau;12\tau)^3}{\vartheta(2\tau;6\tau)^3}\frac{\vartheta(z;\tau)\vartheta(2z+\tau;2\tau)}{\vartheta(4z;4\tau)}.
\end{align*}
We note that the authors of \cite{kimlimlovejoy} use a different normalization for the Appell sum (denoted by $m$) and the $\vartheta$-function (denoted by $j$). With this decomposition, we can begin our study of the growth of the generating function near $\tau = 0$.

\subsection{Odd balanced with $z =0$.}\label{zeq0sec}
\hspace{5mm} In this section, we capture the growth of the of $V(0;q)$ near $\tau=0$, or equivalently, the growth of the coefficients $v(n)$. To do this properly, we take the limit as $z \to 0$. We will find that the the term $T(w,q)$ provides the main estimate. We find and prove the following.
\begin{lemma}\label{v(n)growththeorem}
As $n\to \infty$,
\begin{align*}
    v(n) \sim \frac{1}{16n^{\frac{3}{4}}}e^{\pi\sqrt{n}}\sim \frac{n^{\frac{1}{4}}}{2}\overline{p}(n).
\end{align*}
\end{lemma}
\begin{proof}[Proof of Lemma \ref{v(n)growththeorem}]
Let $M>0$ and let $\textnormal{Im}(\tau)\leq M\textnormal{Re}(\tau)$. The proof amounts to calculating estimates near $\tau =0$ for $T_1(1;q)$, $T(1;q)$ and $T_2(1;q)$ where we have to explicitly calculate
\begin{align*}
    T_2(1;q):= \lim_{w\to 1}T_2(w;q). 
\end{align*}
We begin with calculating the main term for $T(1;q)$, which will turn out to give us the main contribution:
\begin{align*}
    T(1;q) = -\frac{\vartheta\left(\frac{1}{2};\tau\right)}{\vartheta(\tau;2\tau)}\mu\left(\frac{1}{2},\frac{1}{2};2\tau\right).
\end{align*}
For the $\vartheta$-function we turn to Lemma \ref{thetatestimateprelim}, which gives
\begin{align}\label{thethetaforTatzeq0}
    \frac{\vartheta\left(\frac{1}{2};\tau\right)}{\vartheta(\tau;2\tau)} = \frac{-\frac{1}{\sqrt{-i\tau}}\left(1+O\left(q_0^{\frac{1}{2}}\right)\right)}{-2i\frac{1}{\sqrt{-2i\tau}}q_0^{\frac{1}{16}}\left(1+O\left(q_0^\frac{1}{2}\right)\right)} = -i\frac{\sqrt{2}}{2}q_0^{-\frac{1}{16}}\left(1+O\left(q_0^{\frac{1}{2}}\right)\right).
\end{align}
On the other hand, we write $\mu\left(\frac{1}{2},\frac{1}{2};2\tau\right) = \frac{A_1\left(\frac{1}{2},\frac{1}{2};2\tau\right)}{\vartheta\left(\frac{1}{2};2\tau\right)}$ and use the transformation formula in Eq. \eqref{Appelltransnonnormal} to obtain
\begin{align*}
    \mu\left(\frac{1}{2},\frac{1}{2};2\tau\right) &= \frac{A_1\left(\frac{1}{2},\frac{1}{2};2\tau\right)}{\vartheta\left(\frac{1}{2};2\tau\right)}=  \frac{1}{\tau\vartheta\left(\frac{1}{2};2\tau\right) }e^{\frac{\pi i\left(-\frac{1}{4}\right)}{2\tau}}A_1\left(\frac{1}{4\tau}, \frac{1}{4\tau}; -\frac{1}{2\tau}\right)+  \frac{1}{2i}h(0;2\tau).
\end{align*}
Lemmas \ref{thetatestimateprelim} and \ref{lowerboundmordell} tell us that as $\tau \to 0$
\begin{align}
    \vartheta\left(\frac{1}{2};2\tau\right)\label{thetahalfin0proof}&\ll |\tau|^{-\frac{1}{2}},\\
    \label{hin0proof}h(0;2\tau)& =  1.
\end{align}
Writing $A_1\left(\frac{1}{4\tau}, \frac{1}{4\tau}; -\frac{1}{2\tau}\right)$ as a unilateral sum by swapping $n<0$ for $-n$, we have
\begin{align*}
    A_1\left(\frac{1}{4\tau}, \frac{1}{4\tau}; -\frac{1}{2\tau}\right)\ll q_0^{\frac{1}{8}}.
\end{align*}
Combining this with Eqs. \eqref{thetahalfin0proof} and \eqref{hin0proof} gives
\begin{align}\label{mu2in0proof}
    \mu\left(\frac{1}{2},\frac{1}{2};2\tau\right)= \frac{1}{2i} + O\left(|\tau|^{-\frac{1}{2}}q_0^{\frac{3}{16}}\right).
\end{align}
Combining Eqs. \eqref{thethetaforTatzeq0} and \eqref{mu2in0proof}, and then multiplying by $-1$ gives the estimate for $T(1;q)$ as $\tau \to 0$,
\begin{align}\label{overallTestzequal0}
    T(1;q)\sim \frac{\sqrt{2}}{4}q_0^{-\frac{1}{16}}.
\end{align}
We now show that the estimates coming from $T_1(1;q)$ and $T_2(1;q)$ are negligible.
For $T_1(1;q)$, we can use Eq. \eqref{mu2in0proof} with $\tau = \frac{\tau}{2}$ to obtain
\begin{align*}
    T_1(1;q)\ll 1,
\end{align*}
which shows $T_1(1;q)$ is negligible when compared with $T(1;q)$. We now turn to $T_2(1;q)$, which requires formally taking the limit and using the famous formula
\begin{align*}
   \left[\frac{\partial\vartheta(z;\tau)}{\partial z}\right]_{z=0} = -2\pi \eta(\tau)^3.
\end{align*}
Doing so, we find
\begin{align*}
    \lim_{z\to 0}T_2(w;q) = iq^{\frac{11}{8}}\frac{\vartheta(4\tau;12\tau)^3}{\vartheta(2\tau;6\tau)^3}\frac{\eta(\tau)^3\vartheta(\tau;2\tau)}{\eta(4\tau)^3}.
\end{align*}
Using Lemma \ref{thetatestimateprelim}, we find
\begin{align*}
    T_2(1;q)\ll \frac{q_0^{\frac{1}{8}}}{\sqrt{|\tau|}}.
\end{align*}
We see that our main contribution comes from $T_1(1;q)$. We now apply Theorem \ref{Tauberian} to $\frac{1}{2}T(1;e^{-t})$ with $\lambda = \frac{\sqrt{2}}{8}$, $\alpha =0$, and $A = \frac{\pi^2}{4}$ along with the fact that $\overline{p}(n) \sim \frac{1}{8n}e^{\pi\sqrt{ n}}$. 
\end{proof}

\subsection{Odd balanced for fixed roots of unity}\label{zfixedrootsofunity}
\hspace{5mm} We now consider the growth of $V(w;q)$ where $w$ is a generic root of unity, not equal to $\pm i$ or $-1$. We break the study into four cases by splitting the interval $(0,1)$ into four pieces of length $\frac{1}{4}$.

\subsubsection{Case: $0<z<\frac{1}{4}$}
\hspace{5mm} With this restriction, we are able to prove the following. 
\begin{lemma}\label{lessthan14near0} Let $z<\frac{1}{4}$. Then as $n\to \infty$
\begin{align*}
V(w;q)\sim \frac{\sqrt{2}}{4}\frac{w^{-\frac{1}{2}}}{1+w^{-1}}q_0^{\frac{z^2}{2}-\frac{1}{16}}h(2z;2\tau).
\end{align*}
\end{lemma}
\begin{proof}
We address the Appell sum first:
\begin{align}
\begin{split}\label{mupluginlessthan14}
    \mu\left(2z+\frac{1}{2},\frac{1}{2};2\tau\right) = \frac{1}{\vartheta\left(\frac{1}{2};2\tau\right)}A\left(2z+\frac{1}{2},\frac{1}{2};2\tau\right)
    &=\frac{h(2z;2\tau)}{2i}\\
    &+\frac{e^{\frac{\pi i}{2\tau}(4z^2+\frac{1}{4}+2z)}}{2i\vartheta\left(\frac{1}{2};2\tau\right)}\left( \sum_{n\in\mathbb{Z}}\frac{(-1)^nq_0^{\frac{n^2}{4}}}{1-q_0^{\frac{n}{2}-z-\frac{1}{4}}}\right).
\end{split}
\end{align}
It is not hard to prove that the Mordell integral is of $O(1)$ for all $\tau$. This leaves the sum in the parentheses to deal with. Splitting up the sum into positive and negative $n$, we find
\begin{align*}
    \sum_{n\in\mathbb{Z}}\frac{(-1)^nq_0^{\frac{n^2}{4}}}{1-q_0^{\frac{n}{2}-z-\frac{1}{4}}}= -q_0^{\frac{1}{4}}+O\left(q_0^{\frac{1}{4}+z}\right).
\end{align*}
Plugging this back into Eq. \eqref{mupluginlessthan14}, we find 
\begin{align}
\begin{split}\label{mu2z2t}
     \mu\left(2z+\frac{1}{2},\frac{1}{2};2\tau\right)&= \frac{h(2z;2\tau)}{2i}-\frac{q_0^{\frac{3}{16}-z^2-\frac{z}{2}}}{2i\vartheta\left(\frac{1}{2};2\tau\right)}\left(1+O\left(q_0^z\right)\right)\\
     &= \frac{h(2z;2\tau)}{2i} + \frac{\sqrt{2}}{2i}\sqrt{-i\tau}q_0^{\frac{3}{16}-z^2-\frac{z}{2}}\left(1+O(q_0^{z})\right).
\end{split}
\end{align}
Subbing in $z = \frac{z}{2},\tau = \frac{\tau}{2}$, we also find
\begin{align}\label{T1splitlessthan14}
     T_1(w;q) = -iw^{-\frac{1}{2}}\mu\left(z+\frac{1}{2},\frac{1}{2};\tau\right) = -w^{-\frac{1}{2}}\frac{h(z;\tau)}{2} + \frac{w^{-\frac{1}{2}}}{2}\sqrt{-i\tau}q_0^{\frac{3}{32}-\frac{z^2}{8}-\frac{z}{8}}\left(1+O\left(q_0^{z}\right)\right).
\end{align}
With regard to the $\vartheta$-quotients, we can directly apply Lemma \ref{thetatestimateprelim}. We find that for $z<\frac{1}{2}$,
\begin{align}\label{thetaquotientforT}
    \frac{\vartheta(\frac{1}{2}+z;\tau)}{\vartheta(\tau;2\tau)} = \frac{-\frac{q_0^{\frac{z^2}{2}}}{\sqrt{-i\tau}}\left(1+O\left(q_0^{z+\frac{1}{2}}\right)\right)}{-2i\frac{q_0^{\frac{1}{16}}}{\sqrt{-2i\tau}}\left(1+O\left(q_0^{\frac{1}{2}}\right)\right)} = \frac{\sqrt{2}}{2i}q_0^{\frac{z^2}{2}-\frac{1}{16}}\left(1+O\left(q_0^{\frac{1}{2}}\right)\right),
\end{align}
For $z<\frac{1}{2}$, we have
\begin{align}
\begin{split}\label{partialthetaqoutientforfarright}
    \frac{\vartheta(4\tau;12\tau)^3\vartheta(z;\tau)\vartheta(2z+\tau;2\tau)}{\vartheta(2\tau;6\tau)^3}&= \left(\frac{q^{-\frac{1}{12\cdot 8}}}{\sqrt{2}}\left(1+O\left(q_0^{\frac{1}{12}}\right)\right)\right)^3\frac{iw^{-1}}{\sqrt{2}(-i\tau)}\\
    &\;\;\;\;\times q_0^{\frac{3}{2}z^2-z+\frac{3}{16}}\left(1+O(q_0^z)\right)\\
    &=\frac{iw^{-1}}{-4i\tau}q_0^{\frac{5}{32}+\frac{3}{2}z^2-z}\left(1+O\left(q_0^{\min\left(z,\frac{1}{12}\right)}\right)\right).
\end{split}
\end{align}
Finally for $z<\frac{1}{4}$, we have
\begin{align}\label{theta4z4t}
    \vartheta(4z;4\tau) = -\frac{q_0^{\frac{16z^2}{8}-\frac{4z}{8}+\frac{1}{32}}}{2\sqrt{-i\tau}}\left(1+O\left(q_0^{z}\right)\right).
\end{align}
Combining Eqs \eqref{mu2z2t} and \eqref{thetaquotientforT} gives
\begin{align}\label{Tinlessthan14}
    T(w,q) = \frac{\sqrt{2}}{4}w^{-\frac{1}{2}}q_0^{\frac{z^2}{2}-\frac{1}{16}}h_2\cdot\left(1+O\left(q_0^{\frac{1}{2}}\right)\right)+\frac{w^{-\frac{1}{2}}}{2}\sqrt{-i\tau}q_0^{\frac{1}{8}-\frac{z^2}{2}-\frac{z}{2}}\left(1+O(q_0^z)\right),
\end{align}
where $h_2:= h(2z;2\tau)$. Combining Eqs. \eqref{partialthetaqoutientforfarright} and \eqref{theta4z4t} gives,
\begin{align}\label{farrightlessthan14}
    T_2(w;q)  = \frac{w^{-\frac{3}{2}}}{2\sqrt{-i\tau}}q_0^{\frac{1}{8}-\frac{z^2}{2}-\frac{z}{2}}\left(1+O\left(q_0^{\min\left(z,\frac{1}{12}\right)}\right)\right).
\end{align}
We now study the following polynomials on the interval $\left(0,\frac{1}{4}\right)$, which correspond to the exponents of $q_0$ in Eqs. \eqref{T1splitlessthan14}, \eqref{Tinlessthan14}, and \eqref{farrightlessthan14}:
\begin{align*}
    f_1(z)&:= \frac{3}{32}-\frac{z^2}{8}-\frac{z}{8},\;\;f_2(z):=\frac{z^2}{2}-\frac{1}{16},\;\;\textnormal{and}\;\; f_3(z):= \frac{1}{8}-\frac{z^2}{2}-\frac{z}{2}=:f_4(z).
\end{align*}
It is not hard to see that $f_2(z)$ gives the smallest values on the interval $0<z<\frac{1}{4}$, and thus tells us that $T$ gives the primary estimate and $T_1$ and $T_2$ are error terms.\end{proof}

\subsubsection{Case: $\frac{1}{4}<z<\frac{1}{2}$}
\hspace{5mm} We see in this interval, that as before, $T$ will provide the main term. More precisely, we prove the following.

\begin{lemma}\label{growthbetween14and12}
Let $\frac{1}{4}<z<\frac{1}{2}$. Then as $n\to \infty$
\begin{align*}
    V(w;q) \sim -\frac{w^{-\frac{1}{2}}}{1+w^{-1}}\frac{\sqrt{2}}{4}h\left(\frac{3}{4}-z;2\tau\right)q_0^{\frac{z^2}{2}-\frac{1}{16}}.
\end{align*}
\end{lemma}

\begin{proof}
We only need to modify the steps taken to obtain Eqs. \eqref{mu2z2t} and \eqref{theta4z4t}. Let $r:= z-\frac{1}{4}$. Starting with the Appell sum,
we have
\begin{align*}
    \mu\left(2z+\frac{1}{2},\frac{1}{2};2\tau\right)&= \mu\left(2\left(\frac{1}{4}+r\right)+\frac{1}{2},\frac{1}{2};2\tau\right) = - \mu\left(2r,\frac{1}{2};2\tau\right)\\
    &= - \frac{A\left(2r,\frac{1}{2};2\tau\right)}{\vartheta\left(\frac{1}{2};2\tau\right)} = \frac{h\left(r-\frac{1}{2};2\tau\right)}{2i}+\frac{e^{\frac{\pi i}{2\tau}(4r^2-2r)}}{2\tau\vartheta\left(\frac{1}{2};2\tau\right)}A\left(\frac{r}{\tau},\frac{1}{4\tau};-\frac{1}{2\tau}\right)\\
    &=-\frac{h\left(r-\frac{1}{2};2\tau\right)}{2i}-\frac{e^{\frac{2\pi i}{\tau}r^2}}{2\tau\vartheta\left(\frac{1}{2};2\tau\right)}\sum_{n\in \mathbb{Z}}\frac{(-1)^nq_0^{\frac{n^2}{4}}}{1-q_0^{-r+\frac{n}{2}}}.
\end{align*}
We now split the sum as before into positive and negative $n$ and expand the denominators in terms of geometric series, noting that $\frac{1}{1-q_0^{-a}} = -q_0^a\left(1+O(q_0^a)\right)$ for $a>0$. Doing so, we find
\begin{align*}
    \mu\left(2z+\frac{1}{2},\frac{1}{2};2\tau\right) &= -\frac{h\left(r-\frac{1}{2};2\tau\right)}{2i}-\frac{e^{2\frac{\pi i}{\tau}r^2}}{2\tau\vartheta\left(\frac{1}{2};2\tau\right)}\left(-q_0^{r}+O\left(q_0^{2r}\right)\right)\\
    &= -\frac{h\left(r-\frac{1}{2};2\tau\right)}{2i}+\frac{q_0^{r-r^2}}{2\tau\vartheta\left(\frac{1}{2};2\tau\right)}\left(1+O\left(q_0^{r}\right)\right).
\end{align*}
Lemma \ref{thetatestimateprelim} tells us that $\vartheta\left(\frac{1}{2};2\tau\right)\ll \frac{1}{\sqrt{|\tau|}}$. Combining this with the fact that $r-r^2>0$ since $0<r<\frac{1}{4}$, we have
\begin{align*}
    \frac{q_0^{r-r^2}}{2\tau\vartheta\left(\frac{1}{2};2\tau\right)}\left(1+O(q_0^{r})\right) \ll O(q_0^{\varepsilon})
\end{align*}
for some $\varepsilon>0$. It is not hard to see  that $0\neq h\left(r-\frac{1}{2};2\tau\right)\ll 1$. Therefore,
\begin{align}\label{mu2z2tbetween14and12}
    \mu\left(2z+\frac{1}{2},\frac{1}{2};2\tau\right) =  -\frac{h\left(r-\frac{1}{2};2\tau\right)}{2i}\left(1+ O\left(\frac{q_0^{r-r^2}}{\sqrt{|\tau|}}\right)\right).
\end{align}
Combining Eq. \eqref{mu2z2tbetween14and12} with Eq. \eqref{thetaquotientforT} gives
\begin{align}\label{Tbetween14and12}
    T(w;q)= -w^{-\frac{1}{2}}\frac{\sqrt{2}}{4}h\left(\frac{1}{2}-r;2\tau\right)q_0^{\frac{z^2}{2}-\frac{1}{16}}\left(1+O\left(\frac{q_0^{r-r^2}}{\sqrt{|\tau|}}\right)+O\left(q_0^{\frac{1}{2}}\right)\right).
\end{align}
Now we turn to the analog of Eq. \eqref{theta4z4t}. We have
\begin{align}\label{theta4z4tbetween14and12}
    \vartheta\left(4z;4\tau\right) &= \vartheta\left(1+4r;4\tau\right) = -\vartheta\left(4r;4\tau\right)
    =\frac{q_0^{2r^2-\frac{r}{2}+\frac{1}{32}}}{2\sqrt{-i\tau}}\left(1+O\left(q_0^{r}\right)\right).
\end{align}
Combining Eqs. \eqref{partialthetaqoutientforfarright} and \eqref{theta4z4tbetween14and12} and making the substitution $r = z-\frac{1}{4}$ gives
\begin{align*}
    T_2(w;q) =- \frac{w^{-\frac{1}{2}}}{2\sqrt{-i\tau}}q_0^{-\frac{1}{8}-\frac{z^2}{2}+\frac{z}{2}}\left(1+O\left(q_0^{\min\left(z-\frac{1}{4},\frac{1}{12}\right)}\right)\right).
\end{align*}
 We now recycle the same estimate from Eq. \eqref{T1splitlessthan14}, and  we consider again the polynomials 
\begin{align*}
    f_1(z):= \frac{3}{32}-\frac{z^2}{8}-\frac{z}{8},\;\; f_2(z):= \frac{z^2}{2}-\frac{1}{16},\;\; \textnormal{and}\;\;f_3(z) := \frac{1}{8}-\frac{z^2}{2}+\frac{z}{2}.
\end{align*}
We see that on for $\frac{1}{4}<z<\frac{1}{2}$ that $f_2(z)$ is the smallest. Swapping $r$ for $z-\frac{1}{4}$ and dividing Eq. \eqref{Tbetween14and12} by $1+w^{-1}$ proves the claim. \end{proof}

\subsubsection{Case: $\frac{1}{2}<z<\frac{3}{4}$}
\hspace{5mm} Since the Appell sum $\mu\left(2z+\tau,\frac{1}{2};2\tau\right)$ is anti-periodic in the shift $z\to z+\frac{1}{2}$, we can recycle many of the estimates from the previous case to find the following.
\begin{lemma}\label{between1234}
Let $\frac{1}{2}<z<\frac{3}{4}$. As $n\to \infty$
\begin{align*}
   V(w;q)\sim -\frac{w^{-\frac{1}{2}}}{1+w^{-1}}\frac{1}{\sqrt{-i\tau}}q_0^{-\frac{z^2}{2}+\frac{z}{2}-\frac{1}{8}}.
\end{align*}
\end{lemma}

\begin{proof}
We let $z:= \frac{1}{2}+r$. Using identical arguments as in Lemmas \ref{lessthan14near0} and \ref{growthbetween14and12}, we have
\begin{align*}
    T_1(w;q) &= -iq^{\frac{1}{8}}w^{-\frac{1}{2}}\mu\left(r+1,\frac{1}{2};\tau\right) =iq^{\frac{1}{8}}w^{-\frac{1}{2}}\mu\left(r,\frac{1}{2};\tau\right) \\
    &= iq^{\frac{1}{8}}w^{-\frac{1}{2}}\left(\frac{h(r,\frac{1}{2})}{2i}+\frac{e^{\frac{\pi i}{\tau}(r^2-r)}}{\tau\vartheta\left(\frac{1}{2};\tau\right)}A\left(\frac{r}{\tau},\frac{1}{2\tau};-\frac{1}{\tau}\right)\right)\\
    &= \frac{1}{2}w^{-\frac{1}{2}}h\left(r;\frac{1}{2}\right)\left(1+O\left(\frac{q_0^{r-\frac{r^2}{2}}}{\sqrt{|\tau|}}\right)\right).
\end{align*}

We now look at $T$:
\begin{align*}
    T(w,q) &= -w^{-\frac{1}{2}}\frac{\vartheta(1+r;\tau)}{\vartheta(\tau;2\tau)}\mu\left(2r+1+\frac{1}{2},\frac{1}{2};2\tau\right) =  -w^{-\frac{1}{2}}\frac{\vartheta(r;\tau)}{\vartheta(\tau;2\tau)}\mu\left(2r+\frac{1}{2},\frac{1}{2};2\tau\right)\\
    &=-w^{-\frac{1}{2}}\frac{\vartheta(r;\tau)}{\vartheta(\tau;2\tau)}\frac{h(2r;2\tau)}{2i}\left(1+O\left(\sqrt{|\tau|}q_0^{\frac{3}{16}-r^2-\frac{r}{2}}\right)\right)\\
    &=-w^{-\frac{1}{2}}\frac{\sqrt{2}}{4}q_0^{\frac{r^2}{2}-\frac{r}{2}+\frac{1}{16}}h(2r;2\tau)\left(1+O\left(q_0^{\min\left(r,\frac{3}{16}-r^2-\frac{r}{2}\right)}\right)\right),
\end{align*}
where we used Lemma \ref{thetatestimateprelim} in the last step. We finally need $T_2$:
\begin{align*}
    T_2(w;q) &= iw^{-\frac{1}{2}}\frac{\vartheta(4\tau;12\tau)^3}{\vartheta(2\tau;6\tau)^3}\frac{\vartheta\left(r+\frac{1}{2};\tau\right)\vartheta\left(2r+1+\tau;2\tau\right)}{\vartheta(4r+2;4\tau)}\\
    &=-iw^{-\frac{1}{2}}\frac{\vartheta(4\tau;12\tau)^3}{\vartheta(2\tau;6\tau)^3}\frac{\vartheta\left(r+\frac{1}{2};\tau\right)\vartheta\left(2r+\tau;2\tau\right)}{\vartheta(4r;4\tau)}.
\end{align*}
We can recycle all of the estimates for these functions from Lemma \ref{lessthan14near0}, with the exception of the function $\vartheta\left(r+\frac{1}{2};\tau\right)$, which we can obtain from Lemma \ref{thetatestimateprelim}. This gives
\begin{align*}
    T_2(w;q) &= -iw^{-\frac{1}{2}}\left(\frac{q_0^{-\frac{1}{12\cdot 8}}}{\sqrt{2}}\left(1+O\left(q_0^{\frac{1}{12}}\right)\right)\right)^3\left(-\frac{q_0^{\frac{r^2}{2}}}{\sqrt{-i\tau}}\left(1+O\left(q_0^{r+\frac{1}{2}}\right)\right)\right)\\
    &\;\;\;\; \times \left(-2\sqrt{-i\tau}q_0^{-2r^2+\frac{r}{2}-\frac{1}{32}}\left(1+O\left(q_0^r\right)\right)\right)\vartheta(2r+\tau; 2\tau)\\
    &=-\frac{w^{-\frac{1}{2}}e^{-2\pi i r}}{2\sqrt{-i\tau}}q_0^{-\frac{r^2}{2}}\left(1+O\left(q_0^{\min\left(\frac{1}{12},r\right)}\right)\right).
\end{align*}
We now compare the polynomials for $0<r<\frac{1}{4}$,
\begin{align*}
    f_1(r):= \frac{r^2}{2}-\frac{r}{2}+\frac{1}{16}\;\; \textnormal{and}\;\;f_2(r):= -\frac{r^2}{2}.
\end{align*}
The values $f_2(r)$ are clearly smaller on this interval, which shows that our main estimate comes from $T_2$, and the contributions from $T_1$ and $T$ are error terms. Changing variables back to $z$ yields the result. \end{proof}

\subsubsection{Case: $\frac{3}{4}<z<1$.}
\hspace{5mm} As in the previous case, we can recycle some estimates to prove the following.
\begin{lemma}\label{between34and1}
Let $\frac{3}{4}<z<1$. Then as $n\to \infty$,
\begin{align*}
    V(w;q) \sim -\frac{w^{-\frac{1}{2}}}{1+w^{-1}}h\left(z-\frac{3}{4};2\tau\right)\frac{\sqrt{2}}{4}q_0^{\frac{z^2}{2}-z+\frac{7}{16}}.
\end{align*}
\end{lemma}

\begin{proof}
Let $r_0:= 1-z$. We look first at $T_1$ with this variable change:
\begin{align*}
    T_1(w;q) &= -iw^{-\frac{1}{2}}\mu\left(1-r_0+\frac{1}{2},\frac{1}{2};\tau\right)= iw^{-\frac{1}{2}}\mu\left(-r_0+\frac{1}{2},\frac{1}{2};\tau\right)= iw^{-\frac{1}{2}}\frac{A\left(-r_0+\frac{1}{2},\frac{1}{2};\tau\right)}{\vartheta\left(\frac{1}{2};\tau\right)}\\
    &= \frac{h(-r_0;\tau)}{2i}+\frac{e^{\frac{\pi i}{\tau}(r_0^2+\frac{1}{4}-r_0)}}{2i\vartheta\left(\frac{1}{2};\tau\right)}\left( \sum_{n\in\mathbb{Z}}\frac{(-1)^nq_0^{\frac{n^2}{2}}}{1-q_0^{n+\frac{r_0}{2}-\frac{1}{4}}}\right)\\
    &=\frac{h(-r_0;\tau)}{2i} +O\left(q_0^{\frac{1}{8}-\frac{r_0^2}{2}}\right),
\end{align*}
where in the last step, we used the usual trick of splitting the bilateral sum into negative and positive parts and recombined the two into a unilateral sum over positive $n$.

Next comes $T$. The $\vartheta$-functions are standard, and come directly from from Lemma \ref{thetatestimateprelim}. The Appell sum comes from Eq. \eqref{mu2z2tbetween14and12}, which we see by making the substitution $z = r+\frac{3}{4}$, so that $0<r<\frac{1}{4}$. This gives
\begin{align*}
    T(w;q)&= -w^{-\frac{1}{2}}\frac{\vartheta\left(\frac{1}{2}+1-r_0;\tau\right)}{\vartheta(\tau;2\tau)}\mu\left(2r+2,\frac{1}{2};2\tau\right)= w^{-\frac{1}{2}}\frac{\vartheta\left(\frac{1}{2}-r_0;\tau\right)}{\vartheta(\tau;2\tau)}\mu\left(2r,\frac{1}{2};2\tau\right)\\
    &= w^{-\frac{1}{2}}\left(-i\frac{\sqrt{2}}{2}q_0^{\frac{z^2}{2}-z+\frac{7}{16}}\left(1+O\left(q_0^{z-\frac{1}{2}}\right)\right)\right)\\
    &\;\;\;\;\times\left(\frac{h\left(z-\frac{3}{4};2\tau\right)}{2i}\left(1+O\left(|\tau|^{-\frac{1}{2}}q_0^{\left(z-\frac{3}{4}\right)-\left(z-\frac{3}{4}\right)^2}\right)\right)\right)\\
    &=-w^{-\frac{1}{2}}h\left(z-\frac{3}{4};2\tau\right)\frac{\sqrt{2}}{4}q_0^{\frac{z^2}{2}-z+\frac{7}{16}}\left(1+O\left(|\tau|^{-\frac{1}{2}}q_0^{\left(z-\frac{3}{4}\right)-\left(z-\frac{3}{4}\right)^2}\right)\right).
\end{align*}
We finally study $T_2$ using the variable changes $r_1 := z-\frac{1}{2}$, so that $r_1<\frac{1}{2}$ and the previous change $r_0:= 1-z$. We apply Lemma \ref{thetatestimateprelim} which gives
\begin{align*}
    T_2(w;q) &= iw^{-\frac{1}{2}}\frac{\vartheta(4\tau;12\tau)^3}{\vartheta(2\tau;6\tau)^3}\frac{\vartheta\left(z;\tau\right)\vartheta\left(2r_1+1+\tau;2\tau\right)}{\vartheta(4r_1+2;4\tau)}\\
    &= -iw^{-\frac{1}{2}}\frac{\vartheta(4\tau;12\tau)^3}{\vartheta(2\tau;6\tau)^3}\frac{\vartheta\left(z;\tau\right)\vartheta\left(2r_1+\tau;2\tau\right)}{\vartheta(4r_1;4\tau)}\\
    &= \frac{iw}{2\sqrt{-i\tau}}q_0^{-\frac{z^2}{2}+\frac{3z}{2}-\frac{7}{8}}\left(1+O\left(q_0^{1-z}\right)\right).
\end{align*}
As before, we consider the polynomials
\begin{align*}
    f_1(z):= \frac{z^2}{2}-z+\frac{7}{16}\;\; \textnormal{and}\;\; f_2(z):= -\frac{z^2}{2}+\frac{3z}{2}-\frac{7}{8}.
\end{align*}
The function $f_1$ takes strictly negative values and is smaller on the interval $\frac{3}{4}<z<1$, which tells us that $T$ gives us the main estimate and $T_1$ and $T_2$ amount to error terms, which completes the proof.\end{proof}

\section{Proofs of Theorem \ref{maintheorem} and  Corollary \ref{logconcave}}\label{proofsection}
\hspace{5mm} The proof of Theorem \ref{maintheorem} amounts to showing that Lemma \ref{v(n)growththeorem} provides the dominant term in Eq. \eqref{orthogonality}, since the $v(a,c;n)$ are weakly increasing in $n$.
\begin{proof}[Proof of Theorem \ref{maintheorem}]
We can prove the result by recalling in the proof of Lemma \ref{v(n)growththeorem}, that the main term came from $T(1;q)$, and in particular, Eq. \eqref{overallTestzequal0}. This amounts to showing that there exists $\beta>0$ such that each of the estimates in Lemmas \ref{lessthan14near0}--\ref{between34and1} is bounded by $q_0^{\beta-\frac{1}{16}}$. This is not difficult to see by considering the polynomials of the main terms in each of these lemmas on their corresponding intervals for $z$:
\begin{align*}
&\frac{z^2}{2}-\frac{1}{16} \;\;\left(\textnormal{for}\; 0<z<\frac{1}{4}\; \textnormal{and}\; \frac{1}{4}<z<\frac{1}{2}\right),\\
&\frac{z}{2}-\frac{z^2}{2}-\frac{1}{8}\;\;\left(\textnormal{for}\; \frac{1}{2}<z<\frac{3}{4}\right),\\
&\frac{z^2}{2}-z+\frac{7}{16}\;\;\left(\textnormal{for}\; \frac{3}{4}<z<1\right).
\end{align*}
All of these polynomials are strictly bounded below by the critical value $-\frac{1}{16}$, which completes the proof.\end{proof}
We now prove Corollary \ref{logconcave}.
\begin{proof}[Proof of Corollary \ref{logconcave}]
Theorem \ref{maintheorem} and Lemma \ref{v(n)growththeorem} imply that for sufficiently large $n$
\begin{align*}
    \frac{v(a,c;n+1)v(a,c;n-1)}{v(a,c,2n)} = \frac{v(n-1)v(n+1)}{c\cdot v(2n)} = \frac{(2n)^{\frac{3}{4}} e^{\pi(\sqrt{n+1}+\sqrt{n-1}-\sqrt{2n})} }{16 c (n^2-1)^{\frac{3}{4}}}\geq 1,
\end{align*}
where the last step follows from the fact that for positive integers $\sqrt{a+b}<\sqrt{a}+\sqrt{b}$.\end{proof}

\section{Conclusion and the fringe cases for the modulus $c$}\label{conclusionandfringecases}
\hspace{5mm} We showed in this work that the odd-balanced unimodal sequences defined in \cite{kimlimlovejoy} with rank congruent to $a\pmod{c}$, $c\neq 1$ odd, are asymptotically equidistributed about a fourth root of $n$ times the overpartition function $\overline{p}(n)$. We did so by using the representation of $V(w;q)$ as a mixed mock modular form. This representation has its flaws in that the representation breaks down for $z=\frac{1}{2}$. This forces use to throw out all even $c$ since we have to sum over all residue classes in Eq. \eqref{orthogonality}. However, evaluating $V(w;q)$ for example at $w=\pm i$ is no problem as the authors of \cite{kimlimlovejoy} showed that 
\begin{align*}
    V(\pm i;q) = q^{-1}A(q),
\end{align*}
where $A(q)$ is a second order mock theta function. The function $A(q)$ can be easily exponentially bounded by $q_0^{-\frac{1}{16}}$ using its representation as an Appell sum (see the appendix of \cite{Bringmannbook} or \cite{splittingmort}). Thus, if a more manageable $q$-series representation for $V(-1;q)$ can be exploited, we expect that our main result in Theorem \ref{maintheorem} holds for all moduli $c$.

\printbibliography
\end{document}